\documentclass[12pt]{amsart}
\usepackage{bbm}
\usepackage{graphicx}
\usepackage{mathrsfs}
\usepackage{amsmath, amssymb, amsthm, latexsym}
\usepackage{amssymb,amscd,amsmath}
\usepackage{latexsym}
\usepackage{MnSymbol}
\usepackage{enumitem}
\usepackage[center]{caption}
\usepackage{tikz}
\usepackage[all]{xy}
\usepackage{amsxtra,mathtools,stmaryrd}
\newcounter{braid}
\newcounter{strands}

\DeclareMathAlphabet{\bsf}{OT1}{cmss}{bx}{n}

\pgfkeyssetvalue{/tikz/braid height}{1cm}
\pgfkeyssetvalue{/tikz/braid width}{1cm}
\pgfkeyssetvalue{/tikz/braid start}{(0,0)}
\pgfkeyssetvalue{/tikz/braid colour}{black}
\pgfkeys{/tikz/strands/.code={\setcounter{strands}{#1}}}

\makeatletter
\def\cross{%
  \@ifnextchar^{\message{Got sup}\cross@sup}{\cross@sub}}

\def\cross@sup^#1_#2{\render@cross{#2}{#1}}

\def\cross@sub_#1{\@ifnextchar^{\cross@@sub{#1}}{\render@cross{#1}{1}}}

\def\cross@@sub#1^#2{\render@cross{#1}{#2}}

\def\render@cross#1#2{
  \def\strand{#1}
  \def\crossing{#2}
  \pgfmathsetmacro{\cross@y}{-\value{braid}*\braid@h}
  \pgfmathtruncatemacro{\nextstrand}{#1+1}
  \foreach \thread in {1,...,\value{strands}}
  {
    \pgfmathsetmacro{\strand@x}{\thread * \braid@w}
    \ifnum\thread=\strand
    \pgfmathsetmacro{\over@x}{\strand * \braid@w + .5*(1 - \crossing) * \braid@w}
    \pgfmathsetmacro{\under@x}{\strand * \braid@w + .5*(1 + \crossing) * \braid@w}
    \draw[braid] \pgfkeysvalueof{/tikz/braid start} +(\under@x pt,\cross@y pt) to[out=-90,in=90] +(\over@x pt,\cross@y pt -\braid@h);
    \draw[braid] \pgfkeysvalueof{/tikz/braid start} +(\over@x pt,\cross@y pt) to[out=-90,in=90] +(\under@x pt,\cross@y pt -\braid@h);
    \else
    \ifnum\thread=\nextstrand
    \else
     \draw[braid] \pgfkeysvalueof{/tikz/braid start} ++(\strand@x pt,\cross@y pt) -- ++(0,-\braid@h);
    \fi
   \fi
  }
  \stepcounter{braid}
}

\tikzset{braid/.style={double=\pgfkeysvalueof{/tikz/braid colour},double distance=1pt,line width=2pt,white}}

\newcommand{\braid}[2][]{%
  \begingroup
  \pgfkeys{/tikz/strands=2}
  \tikzset{#1}
  \pgfkeysgetvalue{/tikz/braid width}{\braid@w}
  \pgfkeysgetvalue{/tikz/braid height}{\braid@h}
  \setcounter{braid}{0}
  \let\sigma=\cross
  #2
  \endgroup
}
\makeatother

\input xypic
\newtheorem{theorem}{Theorem}%[section]
\newtheorem{proposition}[theorem]{Proposition}

\newtheorem{lemma}[theorem]{Lemma}

\def\Z{\mathbb{Z}}

\def\C{\mathbb{C}}

\def\R{\mathbb{R}}
\def\C{\mathbb{C}}

\def\N{\mathbb{N}}

\def\qed{\hfill$\square$\medskip}

\def\Zpk{\mathbb{Z}/p^{k}}
\def\Zpk1{\mathbb{Z}/p^{k-1}}

\newcommand{\rref}[1]{(\ref{#1})}

\newcommand{\beg}[2]{\begin{equation}\label{#1}#2\end{equation}}
\def\r{\rightarrow}

\def\sl2{\widetilde{SL_{2}(\Z)}}

\author{P.Hu, I.Kriz and Y.Lu}
\title[$\Sigma_3$-cobordism]{Coefficients of the $\Sigma_3$-equivariant
complex cobordism ring}
\thanks{Kriz acknowledges the support of a Simons Collaboration Grant.}
\begin{document}
\maketitle

\begin{abstract}
In this paper, we calculate the coefficient ring of equivariant Thom complex cobordism for the symmetric
group on three elements. We also make some remarks on general methods of calculating certain pullbacks
of rings which typically occur in calculations of equivariant cobordism. 
\end{abstract}

\section{Introduction}

Calculations of equivariant cobordism rings of a finite or compact Lie group $G$ are a current frontier
of equivariant homotopy theory. The cobordism rings we have in mind here are homotopy rings of the corresponding
equivariant Thom spectra, as defined in \cite{lms}. Because of transversality issues, it has been known
since the 60's that these homotopical cobordism rings are different from geometric cobordism rings,
which are actual cobordism rings of weakly stably complex $G$-manifolds. While geometric complex cobordism
rings are of geometric interest, homotopy cobordism rings are of 
much more fundamental homotopy-theoretical interest. The reason is that the equivariant Thom spectra $MU_G$,
which represent them, are proving to play the same key role as complex cobordism plays in 
non-equivariant homotopy theory. Non-equivariantly, it has been known since the 60's by results of Milnor,
Novikov and Quillen that complex cobordism is, in a very strong sense, a universal complex-oriented spectrum,
its associated formal group law being Lazard's universal formal group law. 
Later, methods were developed for constructing a menagerie of complex-oriented spectra from complex
cobordism, and much of the research in stable homotopy theory even today is the study of spectra constructed
by other methods (notably homotopy limits) out of these ingredients.

\vspace{3mm}
Equivariantly, an analogous picture has been much slower to emerge, yet is now clearly coming
into focus. For abelian groups $G$, much progress has been made in calculating the homotopy cobordism
rings $(MU_G)_*$ (\cite{kriz,sinha, ak, strickland, hu}). A theory of $G$-equivariant formal group
laws for $G$ abelian has been constructed \cite{cgk}, and further developed in \cite{cgk1, strickmulti,hu, hks}.
Universality of the equivariant formal group law on $(MU_G)_*$ has been conjectured in \cite{cgk1}, and
recently proved for $G=\Z/2$ by Hanke and Wiemeler \cite{hw}. A program of constructing 
complex oriented $G$-equivariant
spectra out of their $G$-equivariant formal group laws was begun in \cite{hks}. The reader should
realize that these are major developments, as equivariant spectra are much less understood than their 
non-equivariant counterparts, and, in any event, both equivariantly and non-equivariantly, constructing
new spectra (i.e. generalized cohomology theory) ``artificially" out of algebraic data is a major milestone
of homotopy theory.

\vspace{3mm}
While the equivariant picture of complex-oriented spectra is beginning to clear up for $G$ abelian, 
almost nothing remains known for $G$ non-abelian. This is related to the fact that almost no calculations
of equivariant cohomology groups (with the exception of K-theory) are known in the non-abelian case. 
Some recent progress has been made in calculating the $RO(G)$-graded coefficients of 
ordinary cohomology for dihedral groups \cite{kl} and for the quaternion group $Q_8$
\cite{lu}. For cobordism, no calculations were previously known,
although a remarkable completion theorem was proved by Greenlees and May \cite{gm}. Accordingly,
no theory of $G$-equivariant formal group laws has been developed to date for $G$ non-abelian.

\vspace{3mm}
The purpose of this paper is to calculate the equivariant cobordism ring $(MU_{\Sigma_3})_*$ of
the symmetric group on three elements. If we believe that the conjecture of \cite{cgk1} will extend
to $G$ non-abelian, we could therefore now {\em define} $\Sigma_3$-equivariant formal group laws as
represented by the ring $(MU_{\Sigma_3})_*$, and try to use this to study their algebraic properties,
and perhaps, hopefully, eventually obtain a better definition.

\vspace{3mm}
The method used to calculate $(MU_{\Sigma_3})_*$ is the most successful method for calculating 
equivariant generalized cohomology so far, namely the method of {\em isotropy separation}. One uses the properties
of the orbit category of $G$ to investigate certain spectra related to an equivariant spectrumm $E$, which 
can be considered its ``building blocks". Usually, the ``geometric fixed point spectrum" $\Phi^GE$ is one
of these building blocks, and the Borel cohomology spectrum $F(EG_+,E)$ is another. Unless $G$
is elementary cyclic, however, more building blocks are needed to connect the dots. 

\vspace{3mm}
In the case of complex cobordism, the geometric fixed points were calculated by tom Dieck \cite{td}.
We begin in Section \ref{s1} below by calculating the $MU$-Borel cohomology for $\Sigma_3$.
Next, it turns out advantageous to calculate the coefficients of $S^{\infty\alpha}\wedge MU_{\Sigma_3}$
where $\alpha$ is the sign representation of $\Sigma_3$. We do this in Section \ref{s2}.
Next, we calculate the ``intermediate Borel cohomology" $F(S(\infty\alpha)_+,MU_{\Sigma_3})$ in
Section \ref{s3}. This is harder, and dependent on recent explicit algebraic computations of $(MU_{\Z/p})_*$
in \cite{hu}, generalizing a previous result of Strickland \cite{strickland} for $p=2$.
Finally, in Section \ref{s4}, we put this together and calculate $(MU_{\Sigma_3})_*$.
In Section \ref{s5}, we make this still more explicit by making certain algebraic observations,
a byproduct of which is, among other things, a geometric interpretation of the rings $(MU_{\Z/p})_*$ for
$p$ prime.

\section{Borel cohomology}\label{s1}

We first compute the coefficients of $\Sigma_3$-equivariant $MU$-Borel cohomology, which is equivalent to 
$MU^*B\Sigma_3$. We have the diagram
\beg{egamma}{\diagram
\Z/3\dto_\lhd\rto_\subset^\nu
& S^1\dto_\lambda\rto^\iota &S^1\times S^1\dto^\mu\\
\Sigma_3\rto^{\subset}& O(2)\rto^\kappa & U(2)
\enddiagram
}
where $\gamma_\R$ is the real $2$-dimensional irreducible representation, $\kappa$ is complexification,
$\lambda$ is the inclusion of a maximal torus, $\iota(z)=(z,z^{-1})$, and $\mu$ is inclusion of the maximal torus.
Passing to $MU$-cohomology of classifying spaces, it is well known that 
$$MU^*(B(S^1\times S^1))=MU_*[[u_+,u_-]]$$
where $u_+,u_-\in MU^2(\C P^\infty)$ are the Euler classes of the two factors. In these terms, $\mu^*$
is injective, and its image consists of
$$u_\alpha=u_++_Fu_-,\; u=u_+u_-.$$
One can of course instead of $u_\alpha$ also use $u_++u_-$, but the advantage of this notation is that
$u=u_\gamma$ can be considered as the Euler class of the identity representation $\gamma$,
while $u_\alpha$ is the Euler class of the determinant representation $\alpha$.

Wilson \cite{wilson} proved that $\kappa^*$ is onto, and in fact
$$MU^*(BO(2))=MU^*[[u_\alpha, u_\gamma]]/([2]u_\alpha, u_\gamma-\widetilde{u}_\gamma).$$
In the image of $\mu^*$, we have 
$$\widetilde{u}_\gamma=i(u_+)i(u_-)$$
where $i$ denotes the formal inverse. It is worth noting that the class $u_\gamma-\widetilde{u}_\gamma$
is divisible by $u_\alpha u_\gamma$ in $MU^*BU(2)$. 
This can be verified topologically because it restricts trivially to $S^1$ via $\lambda$,
and also directly algebraically, since it is equal in $MU_*[[u_+,u_-]]$ to
$$u_+(u_--i(u_+))+(u_+-i(u_-))i(u_+).$$
On the other hand, it is important to know that despite the notation, $\{3\}u_\gamma$ is a power series in
both $u_\alpha$, $u_\gamma$.
When restricting to $\Sigma_3$, we additionally have the relation
\beg{ethree}{0=\{3\}u_\gamma:=[2]u_+[2]u_--u_+u_-,}
since the inclusion $\Sigma_3\subset U(2)$ factors through $\Z/2\wr \Z/3$, where $[3]u_+=[3]u_-=0$.

\begin{theorem}\label{tbc}
The map $\nu^*$ is onto, and we have
\beg{etbc}{MU^*B\Sigma_3=MU_*[[u_\alpha,u_\gamma]]/([2]u_\alpha,\{3\}u_\gamma).}
\end{theorem}

Before proving this, let us make a couple of observations. First of all, we have
$$[2]u_\alpha \equiv 2u_\alpha \mod (u_\alpha^2),$$
$$\{3\}u_\gamma\equiv 3u_\gamma\mod (u_\gamma^2, u_\gamma u_\alpha).$$
Thus, 
$$(\{3\}u_\gamma)u_\alpha-([2]u_\alpha)u_\gamma\equiv u_\alpha u_\gamma
\mod (u_\alpha^2 u_\gamma, u_\alpha u_\gamma^2),$$
and thus, in the relations of \rref{etbc} imply
\beg{euag}{u_\alpha u_\gamma=0.
}
In particular, the relations \rref{etbc} imply the relation
$$u_\gamma-\widetilde{u}_\gamma=0.$$

\vspace{3mm}
To illustrate the low terms of these series, let us write the universal formal group law as
\beg{euinversal}{\begin{array}{l}z+_Fw=z+w+(-x_1)zw+(x_1^2-x_2)(z^2w+zw^2)+\\(-2x_3+2x_1x_2-2x_1^3)(z^3w+zw^3)+(-3x_3+4x_1x_2-4x_1^3)z^2w^2\\+(-x_4+4x_1x_3-3x_1^2x_2+3x_1^4+x_2^2)(z^4w+zw^4)\\+(-2x_4+11x_1x_3-11x_1^2x_2+10x_1^4+3x_2^2)(z^3w^2+z^2w^3)+... \end{array}}
where $x_1=2m_1,$ $x_2=3m_2,$ $x_3=2m_3-4m_1^3,$ $x_4=5m_4$. Then we have
$$\begin{array}{l}u_\gamma-\tilde{u}_\gamma=u_\alpha u_\gamma(-x_1-x_1^2u_\alpha\\
+(-2x_1^3+x_1x_2-x_3)u_\alpha^2+(4x_1^3-4x_1x_2+3x_3)u_\gamma
\\+(-4x_1^4+3x_1^2x_2-3x_1x_3)u_\alpha^3+(8x_1^4-8x_1^2x_2+6x_1x_3)u_\alpha u_\gamma+...),\end{array}$$
the $[2]$-series is given by
$$\begin{array}{l}[2]u=2u-x_1u^2+(2x_1^2-2x_2)u^3+(-8x_1^3+8x_1x_2-7x_3)u^4\\
+(26x_1^4-28x_1^2x_2+8x_2^2+30x_1x_3-6x_4)u^5+...\;.\end{array}
$$
The $[3]$-series is given by
$$\begin{array}{l}
[3]u=3u-3x_1u^2+(9x_1^2-8x_2)u^3+(-51x_1^3+51x_1x_2-39x_3)u^4\\+(261x_1^4-285x_1^2x_2+72x_2^2+279x_1x_3-48x_4)u^5+...
\end{array}
$$
and the $\{3\}$-series is given by
$$\begin{array}{l}
\{3\}u_\gamma=u_\gamma(3-2x_1u_\alpha+(4x_1^2-4x_2)u_\alpha^2+(-9x_1^2+8x_2)u_\gamma\\
+(-16x_1^3+16x_1x_2-14x_3)u_\alpha^3+(56x_1^3-56x_1x_2+42x_3)u_\alpha u_\gamma+...).
\end{array}
$$

\vspace{3mm}
\begin{proof}[Proof of Theorem \ref{tbc}] We have
\beg{eahss}{H^*(B\Sigma_3;\Z)=\Z[u_\alpha,u_\gamma]/(2u_\alpha,3u_\gamma),}
so the Atiyah-Hirzebruch spectral sequence collapses to $E_2$. Since the right hand side maps
to $MU^*B\Sigma_3$ by $\nu^*$, and on the level of associated graded objects with respect
to the Atiyah-Hirzebruch filtration, it induces an isomorphism, it is an isomorphism.

\end{proof}

\vspace{5mm}

\section{The coefficients of $S^{\infty\alpha}\wedge MU$.}\label{s2}

From now on, we will consider $\alpha$ and $\gamma$ as complex representations of $\Sigma_3$. 
We observe that
\beg{esasg}{S^{\infty\alpha}\wedge S^{\infty\gamma}=\widetilde{E\mathscr{F}[\Sigma_3]}
}
where $\mathscr{F}[\Sigma_3]$ is the family of proper subgroups of $\Sigma_3$, and $\widetilde{X}$
denotes the unreduced suspension of $X$.

By \rref{esasg}, $S^{\infty\alpha}\wedge MU_{\Sigma_3}$
is the homotopy pullback of the diagram
\beg{esadiag}{\diagram
&\widetilde{E\mathscr{F}[\Sigma_3]}\wedge MU\dto\\
S^{\infty\alpha}\wedge F(S(\infty\gamma)_+,MU)\rto & \widetilde{E\mathscr{F}[\Sigma_3]}\wedge 
F(S(\infty\gamma)_+,MU)
\enddiagram
}
where $MU=MU_{\Sigma_3}$. The coefficients of the upper right corner were calculated by tom Dieck \cite{td},
and we have
\beg{es3td}{(\widetilde{E\mathscr{F}[\Sigma_3]}\wedge MU)_*
=\Phi^{\Sigma_3}MU_*=MU_*[u_\alpha,u_\alpha^{-1},u_\gamma,u_\gamma^{-1},b_i^\alpha, b_i^\gamma].
}
The elements $b_i^\alpha$ have dimension $2(i-1)$, and the elements $b_i^\gamma$ have dimension $2(i-2)$,
$i\in \N$.

To calculate the coefficients of the lower left corner, we use the fact that $S(\infty\gamma)$ is the homotopy 
pushout of the diagram
\beg{einfgamma}{\diagram
\Sigma_3\times_{\Z/2} E\Z/2\rto\dto & E\Sigma_3\\
\Sigma_3/(\Z/2). &
\enddiagram
}
Attaching a disjoint base point, mapping into $MU_{\Sigma_3}$, smashing with $S^{\infty\alpha}$, and
taking fixed points,
we get that the lower left corner of the diagram \rref{esadiag} is the homotopy pullback of the
diagram
\beg{ellsadiag}{\diagram
& (S^{\infty\alpha}\wedge F((E\Sigma_3)_+,MU))^{\Sigma_3}\dto\\
\Phi^{\Z/2}MU\rto & \widehat{MU}_{\Z/2} 
\enddiagram
}
where $\widehat{E}$ denotes the (fixed points of the) Tate cohomology of an equivariant spectrum $E$.
Given the relation \rref{euag}, however, we conclude that the vertical arrow \rref{ellsadiag} induces
an isomorphism on coefficients, and thus, the fixed points of the lower left corner of \rref{esadiag}
are just $\Phi^{\Z/2}MU_{\Z/2}$, whose coefficients, by \cite{td}, are
$$(\Phi^{\Z/2}MU)_*=MU_*[u_\alpha, u_{\alpha}^{-1},b_i^\alpha, i\in \N].$$
Now the coefficients of the lower right corner of \rref{esadiag} are obtained from the coefficients
of \rref{ellsadiag} by inverting $u_\gamma$, thus, by inverting $u_\gamma$ in the coefficients of the 
lower left corner of \rref{ellsadiag}, which gives $0$. In other words, the coefficients of the lower right
corner of \rref{esadiag} are $0$, and we obtain

\begin{theorem}\label{tsamu}
We have an isomorphism of rings
\beg{etsamu}{(S^{\infty\alpha}\wedge MU_{\Sigma_3})_*\cong
(\Phi^{\Z/2}MU)_*\scriptstyle\prod{} \displaystyle(\Phi^{\Sigma_3}MU)_*.
}
\end{theorem}
\qed

\vspace{5mm}

\section{The coefficients of $F(S(\infty \alpha)_+,MU)$.}\label{s3}

The spectrum $F(S(\infty \alpha)_+,MU)^{\Sigma_3}$ is the homotopy pullback of the diagram
\beg{esoamu}{\diagram
& F(E\Z/2_+,\Phi^{\Z/3}MU)^{\Z/2}\dto\\
F(B\Sigma_3,MU)\rto & F(E\Z/2_+,\widehat{MU}_{\Z/3})^{\Z/2}.
\enddiagram
}
We begin by calculating the top right corner of \rref{esoamu}. We notice that $\Z/2$ acts on 
$(\Phi^{\Z/3}MU)_*$ by a permutation representation, with 
$$\widehat{H}^0(\Z/2, (\Phi^{\Z/3}MU)_*)=MU_*[u_\gamma,u_\gamma^{-1},b_{2i}^\gamma]/2.$$
In this situation, it is formal that the $MU$-Serre spectral sequence collapses, and we have

\begin{proposition}\label{psoamu1}
The coefficient ring $(F(E\Z/2_+,\Phi^{\Z/3}MU)^{\Z/2})_*$ is the pullback of rings
$$\diagram
& MU_*[u_\gamma,u_\gamma^{-1}, b_{2i}^\gamma][[u_\alpha]]/[2]u_\alpha\dto|>>\tip\\
(\Phi^{\Z/3}MU_*)^{\Z/2}\rto & MU_*[u_\gamma,u_\gamma^{-1}, b_{2i}^\gamma]/2
\enddiagram
$$
where $u_\alpha\mapsto 0$ by the vertical arrow. (Note that the lower left corner of this diagram
means the algebraic fixed points in the category of rings.)
\end{proposition}
\qed

Now in the coefficients of the lower right corner of \rref{esoamu}, \rref{euag} is in effect, so we obtain
\beg{esoamu2}{F(E\Z/2_+,\widehat{MU}_{\Z/3})_*=((\widehat{MU}_{\Z/3})_*)^{\Z/2}.
}
For the same reason, $F(B\Sigma_3,MU)_*$ can be rewritten as the (algebraic) pullback of rings
\beg{esoamu3}{\diagram
& (MU^*B\Z/3)^{\Z/2}\dto|>>\tip^{res}\\
MU^*B\Z/2\rto|>>\tip^{res} & MU_*.
\enddiagram
}
(Also note that at the upper right corner, we have algebraic fixed points in the category of rings.)

\vspace{3mm}
But now by commutation of limits, the pullback of rings
\beg{efppull}{\diagram
&((\Phi^{\Z/3}MU)_*)^{\Z/2}\dto\\
(MU^*B\Z/3)^{\Z/2}\rto& ((\widehat{MU}_{\Z/3})_*)^{\Z/2}
\enddiagram
}
is
\beg{efpull1}{((MU_{\Z/3})_*)^{\Z/2}.}
(Again, this means algebraic fixed points in the category of rings.) Thus, we obtain 

\begin{theorem}\label{tsoamu}
The coefficient ring $F(S(\infty\alpha)_+,MU_{\Sigma_3})_*$ is the limit of the diagram of rings
\beg{etsoamu}{\diagram
&& MU_*[u_\gamma,u_\gamma^{-1},b_{2i}^\gamma][[u_\alpha]]/[2]u_\alpha\dto|>>\tip\\
&((MU_{\Z/3})_*)^{\Z/2}\dto|>>\tip^{res}\rto & MU_*[u_\gamma,u_\gamma^{-1},b_{2i}^\gamma]/2\\
MU^*B\Z/2\rto|>>\tip^{res} & MU_* &
\enddiagram
}
\end{theorem}
\qed

\vspace{3mm}
It is worth pointing out that by \cite{kriz,strickland,hu}, the ring $(MU_{\Z/3})_*$ is now completely known, 
and the action of $\Z/2$ is explicit:

\begin{theorem}\cite{hu}\label{ttt}
Let $p$ be a prime. For $1 \leq \alpha \leq p-1$, let $\alpha^{-1}$ be the inverse of $\alpha$ in 
$(\Z/p)^{\times}$. (Namely, we choose 
the representative in $\Z$ such that 
$1 \leq \alpha^{-1} \leq p-1$.) In $\Z$, write $\alpha \cdot \alpha^{-1} = 1+k_{\alpha}p$. 

The ring $(MU_{\Z/p})_{\ast}$ is isomorphic to
\[ MU_{\ast}[u b_{i, j}^{(\alpha)}, \lambda_{\alpha}, q_j\ | \ \alpha \in (\Z/p)^{\times}, i, j \geq 0]/\sim \]
where the relations are: 
\[ b^{(1)}_{0, 0} = u, b_{0, 1}^{(1)} = 1, b_{0, j}^{(1)} = 0 \]
for $j \geq 2$, 
\[ b_{i, j}^{(\alpha)} - a_{i, j}^{(\alpha)} = u b_{i, j+1}^{(\alpha)} \]
where $a_{i, j}^{(\alpha)}$ is the coefficient of $x^iu^j$ in $x +_F [\alpha]u$, 
\[ q_0 = 0, q_j-r_j = uq_{j+1} \]
where $r_j$ is the coefficient of $u^j$ in $[p]u$, and 
\[ \lambda_1 = 1, \lambda_{\alpha} b_{0, 1}^{\alpha} = 1+ k_{\alpha} q_1  .\]
\end{theorem}

\qed

Note that the relations imply that $\lambda_{\alpha}u_{\alpha} = u$ (where $u_{\alpha} = b_{0,0}^{(\alpha)}$).

\vspace{3mm}
One can, in fact, be more explicit about \rref{efpull1}. We will return to this point in Section 
\ref{efixedgen} below.

\vspace{5mm}

\section{The coefficients of $MU_{\Sigma_3}$.}\label{s4}

Now $MU_{\Sigma_3}$ is the homotopy pullback of the diagram
\beg{emuu}{\diagram
& S^{\infty\alpha}\wedge MU\dto\\
F(S(\infty \alpha)_+,MU)\rto & S^{\infty\alpha}\wedge F(S(\infty \alpha)_+,MU).
\enddiagram
}
The coefficients of the upper right and lower left corners are known by Theorem \ref{tsamu}
and Theorem \ref{tsoamu}. The coefficients of the lower right corner are obtained by inverting $u_\alpha$
in \rref{etsoamu}. We see, however, that after inverting $u_\alpha$, the middle row of \rref{etsoamu} becomes
an isomorphism. Also, in \rref{emuu}, the lower leftmost term of \rref{etsoamu} produces
$(MU_{\Z/2})_*$, which is known by Strickland (see Theorem \ref{ttt}).

\vspace{3mm}
Thus, we need to calculate the algebraic pullback of rings corresponding to the uppermost right corner of
\rref{etsoamu}, when taken through the diagram \rref{emuu}. This diagram has the form
\beg{eurc}{\diagram
& MU_*[u_\gamma,u_\gamma^{-1}, u_\alpha, u_\alpha^{-1}, b_i^{\alpha},b_i^\gamma]\dto\\
MU_*[u_\gamma,u_\gamma^{-1},b_{2i}^\gamma][[u_\alpha]]/[2]u_\alpha\rto &
u_\alpha^{-1}MU_*[u_\gamma,u_\gamma^{-1},b_{2i}^\gamma][[u_\alpha]]/[2]u_\alpha.
\enddiagram
}
In particular, we need to calculate the vertical map \rref{eurc}. As usual, we have
\beg{ebia}{b_i^\alpha\mapsto \text{coeff}_{x^i}(x+_F u_\alpha).
}
Thus, we need to determine where the elements $b_{2i+1}^\gamma$ map.
To this end, we consider the $MU^*$-cohomology of $B\Z/2\wr S^1$. Writing
$$MU^*(S^1\times S^1)=MU_*[[u_+,u_-]],$$
the Serre spectral sequence collapses to $E_2$. Denoting the Euler class of the map $\Z/2\wr S^1\r \Z/2$
by $w$, we need a relation of the form
$$0=w(u_++u_- +HOT).$$
The relation can be detected by inflation associated with the map
\beg{einf}{(\Z/2\ltimes S^1)\times S^1\r \Z/2\wr S^1}
where on the left hand side, an element $\alpha$ of the first copy of $S^1$ maps to
$(\alpha,\alpha^{-1})$, and an element $\beta$ of the second copy of $S^1$ maps to $(\beta,\beta)$.

\begin{lemma}\label{linj1}
This inflation in $MU^*$-cohomology is injective (and also remains injective when
inverting the Euler class $w$ of the projection to $\Z/2$).
\end{lemma}

\begin{proof}
Restricting to the $S^1\times S^1$-subgroups, the inflation on $MU^*$-cohomology can be written as
\beg{errr}{MU_*[[u,v]]\r MU_*[[x,z]]
}
where
$$u\mapsto x+_F z,\; v\mapsto i(x)+_F z$$
where $i(x)$ is the formal inverse. Clearly, this is injective. To deduce the statement of the Lemma, 
by the Serre spectral sequence (which collapses both in the source and the target), it 
suffices to show that \rref{errr} induces an injection on the $\Z/2$-Tate cohomology $\widetilde{H}{\Z/2}$. 
However, explicitly, on $\widetilde{H}{\Z/2}$, we get the map
$$MU_*[[uv]]\mapsto MU_*[[xi(x),z]]
$$
where
$$uv\mapsto (x+_Fz)(i(x)+_Fz),$$
which is injective. 
\end{proof}

Now we know the cohomology of the
classifying space of the source of \rref{einf}. Explicitly, since the first factor is $O(2)$, we can write
$$MU^*B((\Z/2\ltimes S^1)\times S^1)=MU_*[[u_\alpha,u_\gamma,z]]/([2]u_\alpha, u_\gamma-
\widetilde{u}_\gamma).$$
The inflation map is
$$w\mapsto u_\alpha,\; u_+\mapsto u_++_F z,\; u_-\mapsto u_-+_F z.$$
Further, in the target of the inflation, we have the relation
$$(u_++_F z)(u_-+_Fz)-(i(u_+)+_Fz)(i(u_-)+_F z).$$
Thus, in $MU^*B(\Z/2\wr S^1)$, we obtain the relation
\beg{eupm}{u_+u_--(u_++_Fw)(u_-+_Fw),}
which is of the required form. Thus, we have proved

\begin{theorem}\label{twr}
The ring $MU^*B(\Z/2\wr S^1)$ is isomorphic to the quotient of $MU_*[[u_\alpha,u_\gamma,w]]/[2]w$
by the relation \rref{eupm} (where, as usual, we write $u_\alpha=u_++_F u_-$, $u_\gamma=u_+u_-$).
\end{theorem}
\qed

To see what this has to do with the elements $b_{2i+1}^\gamma$ in \rref{eurc}, we note that 
we have 
$$MU_*[u_\gamma,u_\gamma^{-1}, u_\alpha, u_\alpha^{-1},b_i^\alpha, b_i^\gamma]=
\Phi^{\Z/2}(MU\wedge (BU\times BU)_+)[u_\gamma,u_\gamma^{-1}]$$
where $\Z/2$ acts by interchanging the $BU$ coordinates. By the same method as in \cite{kriz}, we then
obtain a map from this ring to $MU^*B(\Z/2\wr S^1)$ given by
\beg{eass}{u_\alpha\mapsto w,\; b_i^\alpha\mapsto \text{coeff}_{x^i}(x+_F w),\;
u_\gamma\mapsto u_\gamma,\; b_{i}^\gamma\mapsto \text{coeff}_{x^i}(x+_Fu_+)(x+_Fu_-).}
In more detail, start with the composition
\beg{egeomtate}{MU\wedge BU\r (\Phi^{S^1}MU)^{S^1}\r(\widehat{MU}_{S^1})^{S^1}
}
where we work over the universe containing only the trivial and standard representation of $S^1$, and
the first map \rref{egeomtate} comes from the tom Dieck calculation \cite{td}. Smashing two copies of 
\rref{egeomtate}, we obtain a (naively) $\Z/2$-equivariant map, and taking its homotopy fixed points 
gives the required map \rref{eass}. (To be even more precise, in the target, we have to compose 
with another map, completing, on the level of Borel cohomology, the smash product, and then 
inverting the Euler classes.)

\begin{lemma}\label{linj2}
The map \rref{eass} is injective. 
\end{lemma}

\begin{proof}
The proof proceeds in the same way as the proof of Lemma \ref{linj1}, once we prove that the
map induced by \rref{egeomtate} on coefficients is injective. This map is 
$$ b_i\mapsto \text{coeff}_{x^i}(x+_Fu)\in MU_*[[u]].$$
We must show that the images of the $b_i$'s are algebraically independent. To this end, note that
for $i\geq 1$, the lowest term of the power series in $u$ to which $b_i$ maps is $a_{1,i}u$. If there is
an algebraic relation between these elements, it must remain valid after dividing by $u$. But if those elements
are algebraically dependent, they are also algebraically dependent modulo $u$, which means that the $a_{1,i}$'s
are algebraically dependent over $MU_*$. This is well known not to be the case. (In fact, 
the coefficients of the series
$$\int_0^x\frac{dt}{\displaystyle \sum_{i\geq 0}a_{i1}t^i}$$ 
are the 
coefficients of the universal logarithm, which are algebraically independent by Lazard's theorem.)
\end{proof}

To see what happens to the $b_{i}^\alpha$'s, we can enhance the relation
\rref{eupm} by adding formally another formal variable $t$, thus obtaining
\beg{eupm1}{(u_++_F t)(u_-+_F t)-(u_++_Fw+_F t)(u_-+_F w+_F t).
}
Note that by the Serre spectral sequence, this relation must in fact follow fromo \rref{eupm}, but it is 
more convenient for our purposes. In effect, translating back via \rref{eass}, we obtain
\beg{eass1}{\sum_{j\geq 1}b_j^\gamma t^j=\sum_{j\geq 1}b_j^\gamma(u_\alpha+_F t)^j,}
valid in the bottom right term of \rref{eurc}. Note that examining the $t^{j-1}$ coefficient of \rref{eass1},
and using $[2]u_\alpha=0$, for $j$ odd, we obtain an expression containing a summand of 
$b_j^\gamma u_\alpha$ and possibly $b_k^\gamma u_\alpha$ with some additional coefficients for
$k<j$, modulo higher powers of $u_\alpha$. Working by induction on $j$, we can eliminate the summands
$b_k^\gamma u_\alpha$ with $k$ odd modulo higher powers of $u_\alpha$, and then repeat the procedure,
ultimately expressing $b_j^\gamma u_\alpha$ as a power series in
$u_\alpha$ (in powers $\geq 2$) whose coefficients are polynomials in the $b_i$'s with $i$
even. In particular, this
gives recursive relations in the lower right corner of \rref{eurc} of the form
\beg{eass2}{b_{2i+1}^\gamma=\sum_{j\geq 1} c_j u_\alpha^j,
}
where $c_j$ are polynomials with coefficients in the $b_{2k}^\gamma$'s. 

For example, we have
$$\begin{array}{l}b_1^\gamma=-b_2^\gamma u_\alpha-2x_1b_2^\gamma u_\alpha^2\\
+(-6x_1^2b_2^\gamma-3x_2b_2^\gamma+5b_4^\gamma)u_\alpha^3+
\\(-40x_1^3b_2^\gamma+x_1x_2b_2^\gamma-43x_3b_2^\gamma+54x_1b_4^\gamma)u_\alpha^4+...,
\end{array}$$
also
$$\begin{array}{l}
b_3^\gamma=2x_1b_2^\gamma+(6x_1^2b_2^\gamma+3x_2b_2^\gamma-6b_4^\gamma)u_\alpha
\\+(42x_1^3b_2^\gamma-3x_1x_2b_2^\gamma-33x_3b_2^\gamma-58x_1b_4^\gamma)u_\alpha^2+...,
\end{array}
$$
and
$$
b_5^\gamma=-2x_1^3b_2^\gamma+2x_1x_2b_2^\gamma+4x_3b_2^\gamma+4x_1b_4^\gamma+...\; .
$$

This lets us calculate
the pullback of rings \rref{eurc} by the same method as in \cite{hu,strickland}. The answer is the 
ring
\beg{er}{\begin{array}{l}R=MU_*[u_\gamma,u_\gamma^{-1}, u_\alpha,b_{i,j}^\alpha, q_j,
b_{2i}^\gamma, b_{2i+1,j}^\gamma]/\\
(b_{i,j}^\alpha-a_{ij}=b_{i,j+1}^\alpha u_\alpha, q_0u_\alpha, q_j-r_j=q_{j+1}u_\alpha,
b_{2i+1,j}^\gamma-c_j=b_{2i+1,j+1}^\gamma u_\alpha)
\end{array}}
Note that this maps canonically to $MU_*[u_\gamma,u_\gamma^{-1},b_{2i}^\gamma]/2$
by mapping via the vertical arrow in \rref{eurc} (which we determined explicitly), and taking the constant
term of the applicable $u_\alpha$-series. In summary, we have our main result:

\begin{theorem}\label{tmain}
The ring $(MU_{\Sigma_3})_*$ is the limit of the diagram of rings
\beg{etmain}{\diagram
&&R\dto|>>\tip\\
&((MU_{\Z/3})_*)^{\Z/2}\dto|>>\tip^{res}\rto & MU_*[u_\gamma,u_\gamma^{-1},b_{2i}^\gamma]/2\\
(MU_{\Z/2})_*\rto|>>\tip^{res} &MU_*
\enddiagram
}
where the rightmost vertical arrow is described above.
\end{theorem}
\qed

\vspace{5mm}

\section{A method for computing certain pullbacks of rings}\label{s5}
\label{efixedgen}

In this section, we discuss how one can find generators of the types of rings which occur in equivariant 
cobordism as pullbacks of ``Tate diagrams". All rings considered are commutative associative unital rings.
Our main result is the following

\vspace{3mm}

\begin{theorem}\label{tpulll}
Let $h:R\r S$ be a homomoprhism of rings and let $h(u)=v$ where $u\in R$ is a regular element
(=non-zero divisor). Consider the induced homomorphisms
$$\overline{h}:R\r S/(v),$$
$$\phi:u^{-1}R\r v^{-1}S.$$
Denote $I=Ker(\overline{h})$, and let $Q\subseteq R$ be a set of representatives of the cosets $R/I$
and let, for $t\in T$, $q(t)\in Q$ denote the representative of the coset $t+I$.
Let
$T\subseteq R$ be a subset.
Assume
\begin{enumerate}
\item\label{tipull1}
The set $T\cup Q$ generates $R$ as a ring.

\item\label{tipull2}
The set $\{t-q(t)\mid t\in T\}$ generates $I$ as an ideal.

\item\label{tipull3}
The homomorphism $\overline{h}$ is onto.

\end{enumerate}

Then there exist unique elements $t_i\in \phi^{-1}S$, $q_{t,i}\in Q$ for $t\in T$, $i\in \N_0$ such that 
$$t_0=t$$
$$ut_{i+1}=t_i-q_{t,i}.$$
Additionally, the set
\beg{epulls}{Q\cup \{t_i\mid t\in T, i\in \N_0\}
}
generates the ring $\phi^{-1}S\subseteq u^{-1}R$.

\end{theorem}

\begin{proof}
Denoting by $J$ the kernel of the induced homomorphism
$$\widetilde{h}:\phi^{-1}S\r S/(v),$$
we have, by \rref{tipull3}, an isomorphism $R/I\cong \phi^{-1}S/J$. Thus,
for $t\in \phi^{-1}S$, there is a unique $q(t)\in Q$ such that $t-q(t)\in J$ (extending our existing notation).
We may then put 
$$q_{t,i}=q(t_i),$$
thus proving existence. In fact, we now see that for every $t\in \phi^{-1}S$, there exists a unique
element $t^{\prime}\in \phi^{-1}S$ such that
$$ut^\prime=t-q(t).$$
Uniqueness follows from the fact that $Q$ is a set of representatives of $R/I$, and $u$ is a regular element.

To prove that the set \rref{epulls} generates the ring $\phi^{-1}(S)$, let $x\in \phi^{-1}(S)\subseteq u^{-1}R$. 
Then, by \rref{tipull1},  there exist an $n\in \N_0$ such that
$$u^nx=p(q,t_i)$$
where the right hand side denotes some polynomial in the various elements $t_i$, $t\in T$ and $q\in Q$. Our job
is to prove that $n=0$. Assume, then, $n>0$. Then 
$$v\mid p(\phi(q),\phi(t_i)),$$
and thus, 
$$p(\widetilde{h}(q),\widetilde{h}(t_i))=0\in S/(v).$$
Therefore, 
$$p(q,t_i)\in J,$$
and hence
$$p(q,q_{t,i})\in J.$$
But we also have $p(q,q_{t,i})\in R$, and $J\cap R=I$ (since $(\phi^{-1}S)/J\cong R/I\cong S/(v)$),
so 
$$p(q,q_{t,i})\in I=(t-q(t)\mid t\in T)\subseteq (ut_1\mid t\in T)\subseteq (u)\subseteq\phi^{-1}S.$$
(The first equality is by \rref{tipull2}.)
On the other hand, 
$$p(q,t_i)-p(q,q_{t,i})\in (t_i-q_{t,i})=(ut_{i+1})\subseteq (u)\subseteq \phi^{-1}S.$$
Therefore,
$$v=p(q,t_i)\in (u)\subseteq \phi^{-1}S$$
which, since $u$ is a regular element, contradicts the minimality of $n$.

\end{proof}

\vspace{3mm}

\noindent
{\bf Example 1} (\cite{hu}) {\bf :} For a prime $p$, to calculate $(MU_{\Z/p})_*$, we may let
\beg{erpull1}{R=MU_*[u_\alpha,b_i^\alpha\mid \alpha=1,\dots,p-1,i\in \N]
[u_\alpha^{-1}u_1\mid \alpha=2,\dots,p-1].}
We put $u=u_1$. Then $u\in R$ is a regular element, and we have
$$u^{-1}R=\Phi^{\Z/p}MU_*.$$
We put 
$$S=MU_*[[v]]/([p]v/v)$$
and define $h:R\r S$ by 
$$h(u_\alpha)=[\alpha]v,$$
$$h(b_i^\alpha)=\text{coeff}_{x^i}(x+_F [\alpha]v).$$
$$ h(u_{\alpha}^{-1}u) = \frac{[\alpha^{-1}]([\alpha]u)}{[\alpha]u} = \sum_{j \geq 0} 
a_{0, j+1}^{(\alpha^{-1})} ([\alpha]u)^j .$$
(Recall here that $\alpha^{-1}$ is the inverse of $\alpha$ in $(\Z/p)^{\times}$.)
By \cite{kriz}, $(MU_{\Z/p})_*$ is the pullback of the Tate diagram
$$\diagram
& u^{-1}R\dto^\phi\\
MU_*[[v]/[p]v\rto &v^{-1}S.
\enddiagram
$$
Now note that the image of the lower horizontal homomorphism is $S$, and its
kernel is
\beg{emuann}{MU_*\{\frac{[p]v}{v}\}.
}
Therefore, additively,
\beg{emuann1}{(MU_{\Z/p})_*=MU_*\{\frac{[p]v}{v}\}\oplus \phi^{-1}(S),}
and Theorem \ref{tpulll} can be used to find the second summand, giving the exact generators listed
in Theorem \ref{ttt}. Somewhat more work \cite{hu} is needed to determine the relations (i.e. the 
extension, and proving that there are no other relations except the ones we already know). 

However, 
even without knowing the relations, in some sense, we can consider \rref{emuann1} to be an
explicit calculation of
$(MU_{\Z/p})_*$, since we have explicit generators of the second summand inside the explicitly given
ring $u^{-1}R$.

\vspace{5mm}
\noindent
{\bf Example 2:} The ring \rref{efpull1} can be computed by the same method. Let us use 
Diagram \rref{efppull}. We have 
\beg{efz3pull1}{\Phi^{\Z/3}MU_{*}=MU_{*}[u_+,u_-,u_+^{-1},u_-^{-1},b_i^+,b_i^-\mid i\in\N]}
where in comparison with \rref{erpull1}, we put $u_+=u_1$, $u_-=u_2$, $b_i^+=b_i^1$, $b_i^-=b_i^2$.
The action reverses the $+$ and $-$. We put
$$\widetilde{S}=(F(E\Z/3_+,MU)_*)^{\Z/2}=(MU_*[[u_+]]/[3]u_+)^{\Z/2}=MU_*[[v]]/\{3\}v,$$
$$S=MU_*[[v]]/(\{3\}v/v).$$
Here $v=u_+u_-=u_+([2]u_+)$, and $\{3\}v$ is the same series as $\{3\}u_\gamma$, with $u_\gamma$
replaced by $v$ and $u_\alpha$ by $0$.

Next, we need to compute the $\Z/2$-fixed points of \rref{efz3pull1}. To this end, recall the computation
of $A=\Z[x_1,y_1,\dots,x_n,y_n]^{\Z/2}$ where the generator of $\Z/2$ switches $x_i$ with $y_i$. 
The key point here is that if we denote by $\sigma_{i,\epsilon}$ the $\epsilon$'th elementary symmetric
polynomial in $x_i, y_i$, $\epsilon=0,1$,
then $A$ is a a free module over $B=\Z[\sigma_{1,0},\sigma_{1,1},\dots \sigma_{n,0},\sigma_{n,1}]$
with basis $x_I=x_{i_1}\dots x_{i_k}$ for $I=\{i_1<\dots <i_k\}\subseteq \{1,\dots,n\}$ (note that it
suffices to prove the case $n=1$, which is easily checked). Now of course $B$ is fixed under
the action, and the action of the generator $\rho$ of $\Z/2$ on $x_i$ is by 
$$x_i\mapsto \sigma_{i,0}-x_i.$$
This means that $A$ has an increasing filteration by $B[\Z/2]$-modules, where the filtration degree of $x_I$ is
the number of elements $|I|$ where $I$ is as above.
Thus, for $|I|$ even, the element
\beg{exiyi}{(1+\rho)(x_I)=x_I+y_I
}
is written in the basis as $2x_I$  plus elements of lower filtration degree, while for $|I|$ odd, $(1+\rho)x_I$ is
entirely in lower filtration degree. In fact, verification shows that for $|I|$ odd,
$$2(1+\rho)x_I=-\sum_{i\in I} \sigma_{i,0}(1+\rho)x_{I\smallsetminus\{i\}}.$$
Thus, $A$ has a sub-$B[\Z/2]$-module generated on $2x_I+LOT$ on which $\rho$ acts by minus. The quotient
is an extended $B[\Z/2]$-module generated by $x_I$. We conclude that the $B$-module $A^{\Z/2}$ is generated
by the elements \rref{exiyi}. 

\vspace{3mm}
We conclude that the ring $(\Phi^{\Z/3}MU_*)^{\Z/2}$ is generated by $u=u_+u_-$, $u^{-1}$,
$\beta_i=b_i^+b_i^-$, and elements of the form 
\beg{eoddgen}{(1+\rho)(x_{i_1}\dots x_{i_k})
}
where $x_{i_j}$ are different elements of the form $u_+$, $b_i^+$.

Now let $R$ be the subring generated by $u$, $\beta_i$ and the elements \rref{eoddgen}. Then we are
again in the situation of Theorem \ref{tpulll}: we have $q(\beta_0)=1$, otherwise
the value of $q$ on the specified generators is $0$. We have, again, $S/(v)=MU_*/3$. Thus,
again, the assumptions of Theorem \ref{tpulll} are satisfied. Again, the
kernel of the projection $\widetilde{S}\r S$ is 
\beg{emumumu}{MU_*\{\frac{\{3\}v}{v}\},}
so $(MU_{\Z/3})_*^{\Z/2}$ is additively a sum of \rref{emumumu} and a factor ring whose generators 
are computed by Theorem \ref{tpulll}.

\vspace{5mm}
\noindent
{\bf Comment:} There is a more geometric and arguably more precise, but less explicit
description of the construction of Theorem \ref{tpulll}. Assume that $(v)\subset S$ is a prime ideal. 
%Let $L$ be the field of fractions of the local ring $S_{(v)}$. Choosing a maximal local subring $V\subset L$
%such that $S_{(v)}\subseteq V$ is an inclusion of local rings (i.e. pullback along it preserves maximal ideals), then
%$V$ is a valuation ring. Let $v:L\r \Gamma$ be the valuation for some ordered abelian group $\Gamma$.
Thus, we obtain a morphism of affine schemes
\beg{especm}{Spec(S)\r Spec(R)}
where the image of the closed point is $I$. Blow up $I$. By universality, \rref{especm} lifts to the blow-up.
The image of the lift is an affine open subset. Blow up the image of $(v)$ under the lift, and repeat the procedure.
The procedure could stop after finitely many steps if the image of $(v)$ after a series of blow-ups
is a principal ideal, which would correspond to all the generators $t_i$ being in $Q$
(which would give $t_{i+1}=0$). However, the experience \cite{strickland, hu}
has been that in the case of equivariant
cobordism rings, typically the procedure does not stop after finitely many steps (i.e. the ideal $(v)$ is not
``divisorial"). Note, of course, that in these examples, we are dealing with coherent, but not Noetherian rings.

\end{document}